\documentclass[11pt]{elsarticle}

\usepackage{amsmath, amssymb}
\usepackage{graphics}
\usepackage{caption}
\usepackage{subcaption}
\usepackage{amsthm}
\usepackage{algorithm}
\usepackage{algorithmic}
\usepackage{color}
\usepackage{cancel}
\usepackage{caption}
\usepackage{subcaption}

\def\deg{{\rm deg}}

\newtheorem{theorem}{{\bf Theorem}}

\newtheorem{corollary}[theorem]{{\bf Corollary}}
\newtheorem{proposition}[theorem]{{\bf Proposition}}
\newtheorem{lemma}[theorem]{{\bf Lemma}}
\newtheorem{example}{{\bf Example}}

\begin{document}

\begin{frontmatter}
\title{On the square-freeness of the offset equation to a rational planar curve}


\author[a]{Juan Gerardo Alc\'azar\fnref{proy,proy2}}
\ead{juange.alcazar@uah.es}
\author[b]{Jorge Caravantes\fnref{proy}}
\ead{jcaravan@mat.ucm.es}
\author[c]{ Gema M. Diaz-Toca\fnref{proy,proy3}}
\ead{gemadiaz@um.es}

\address[a]{Departamento de F\'{\i}sica y Matem\'aticas, Universidad de Alcal\'a,
E-28871 Madrid, Spain}
\address[b]{Departamento de \'Algebra, Universidad Complutense de Madrid, E-28040 Madrid, Spain}
\address[c]{Departamento de Matem\'atica Aplicada, Universidad de Murcia,  E-30100 Murcia, Spain}


\fntext[proy]{Partially supported by the Spanish Ministerio de Econom\'{\i}a y Competitividad and by the European Regional Development Fund (ERDF), under the project  MTM2014-54141-P.}

\fntext[proy2]{Member of the Research Group {\sc asynacs} (Ref. {\sc ccee2011/r34}) }

\fntext[proy3]{Partially supported by the Research Group {\sc E078-04} of the University of Murcia}



\begin{abstract}
It is well known \cite{Far2} that an implicit equation of the offset to a rational planar curve can be computed by removing the extraneous components of the resultant of two certain polynomials computed from the parametrization of the curve. Furthermore, it is also well known that the implicit equation provided by the non-extraneous component of this resultant has at most two irreducible factors \cite{Rafa1}. In this paper, we complete the algebraic description of this resultant by showing that the multiplicity of the factors corresponding to the offset can be computed in advance. In particular, when the parametrization is proper, i.e. when the curve is just traced once by the parametrization, we prove that any factor corresponding to a simple component of the offset has multiplicity 1, while the factor corresponding to the special component, if any, has multiplicity 2. Hence, if the parametrization is proper and there is no special component, the non-extraneous part of the resultant is square-free. In fact, this condition is proven to be also sufficient. Additionally, this result provides a simple test to check whether or not a given rational curve is the offset of another curve. 
\end{abstract}

 \end{frontmatter}
 

\section{Introduction}\label{section-introduction}

{\it Offset curves} have been largely studied in the field of Computer Aided Geometric Design (see \cite{Farouki} for an overview on offset curves). These curves allow to give thickness to a thin object, and have been successfuly used in contexts like manufacturing, graphic design or robotics \cite{Farouki}. Furthermore, different theoretical aspects related to offset curves have been considered in the literature: algebraic properties \cite{Far2}, \cite{Rafa1}, topology \cite{AS07}, \cite{A08}, parametrization \cite{AS97}, singularities \cite{paper}, \cite{Kim}, \cite{Maekawa}, \cite{Seong} or genus computation \cite{AS99}, \cite{Fukushima}, to give a non-exhaustive list of topics or papers. Many of these papers focus on planar curves described by means of rational parametrizations, since this type of parametrizations is popular and widely used in CAGD. 

In the seminal paper \cite{Far2}, a method to find an implicit equation for the offset of a planar rational curve is given. This method, which we briefly summarize in Section \ref{sec-prelim} of this paper, requires to compute the resultant of two polynomials that are built from the parametrization. Afterwards, in certain cases we need to remove an extraneous factor, whose components can be computed a priori. The resulting equation $F(x,y)$ has at most two irreducible components \cite{Rafa1}, which can be of either {\it simple} or {\it special} nature, in the terminology of \cite{Rafa1}. Hence, $F(x,y)$ can be written as $F(x,y)=(f_1(x,y))^r\cdot (f_2(x,y))^s$, where $f_1(x,y)$ and $f_2(x,y)$ are irreducible polynomials, and $f_2(x,y)$ is assumed to be constant in the case when the offset has just one component. However, no observation is made in the literature on the values of $r,s$, i.e. on the multiplicity structure of $F(x,y)$. The goal of this paper is to fill this gap, thereby completing the description of the algebraic structure of the resultant giving rise to the offset of a rational curve.

In this paper we will see that the multiplicities $r,s$ can be computed in advance, and are related to two notions. The first one, which has already been mentioned in the preceding paragraph, is the simple or special nature of the corresponding offset component. The second one is the {\it tracing index} \cite{SWPD} of the parametrization. Essentially, the tracing index describes how many times a rational parametrization traces a given curve, or, in other words, how many times the parametrization traces a regular point of the curve. If the tracing index is 1, then the parametrization is said to be {\it proper}; this means that the degree of the parametrization is minimum. In this paper we prove that if the parametrization is proper then $r=s=1$ whenever there is no special component, and $r=1$, $s=2$ whenever there is a special component; in particular, if we have properness and there is no special component, then $F(x,y)$ is square-free. When the parametrization has tracing index $n> 1$, we get $r=s=n$ if there is no special component, and $r=n$, $s=2n$ if there is a special component.  As a consequence, $F(x,y)$ is square-free iff the parametrization of the curve is proper, and the offset has no special component (see Corollary \ref{corfinal}).

The motivation for this paper came up while the authors were developing the results of \cite{paper}, where the problem of computing the self intersections of an offset curve was studied. Indeed, at a certain point in \cite{paper} (see Remark 2 in \cite{paper}) it was necessary to study whether or not the polynomial $F(x,y)$ was square-free. Furthermore, as we show in this paper (see Theorem \ref{charact}), the square-free character of $F(x,y)$ also characterizes the situation when the offset has no special component, which is a necessary hypothesis, for instance, in \cite{paper} and \cite{Fukushima}. In turn \cite{Rafa1}, this allows to recognize whether or not a given rational curve is the offset to another curve.   

The paper consists of two sections. In Section \ref{sec-prelim} we provide some preliminary definitions and results. In Section \ref{sec-main} we prove the theorems on the multiplicity structure of $F(x,y)$, and we illustrate them with some examples.

\section{Preliminaries and generalities} \label{sec-prelim}

%

Let ${\mathcal C}$ be a real, rational plane curve, not a line or a circle, parametrized by
\begin{equation}\label{curve}
\phi(t)=\left(\frac{p_1(t)}{q_1(t)},\frac{p_2(t)}{q_2(t)}\right),
\end{equation}
where $p_i(t),q_i(t)$, $i=1,2$, are polynomials with rational coefficients and $\gcd(p_1,q_1)=\gcd(p_2,q_2)=1$. It is well-known that the resultant provides the implicit equation of the curve. From Theorem 4.41 in \cite{SWPD}, 
\begin{equation} \label{res}
\mbox{Res}_t(q_1(t) x-p_1(t),q_2(t)y-p_2(t))=c\cdot(f(x,y))^n,
\end{equation}
where $c$ is a nonzero constant, $f(x,y)$ is the implicit equation of ${\mathcal C}$, and $n$ is the \emph{tracing index} of the parametrization $\phi(t)$, i.e. the number of times the parametrization $\phi(t)$ covers a regular point $(x,y)\in {\mathcal C}$. The interested reader can check \S 1.6 \cite{SWPD} for further reading on the notion of tracing index. Additionally, we say that $\phi(t)$ is {\it proper} iff it is birational, i.e. iff $\phi(t)$ is injective except perhaps for finitely many values of the parameter $t$ (corresponding to the self-intersections of the curve). One can prove (Theorem 4.30 in \cite{SWPD}) that $\phi(t)$ is proper iff the tracing index is 1. Since properness can always be achieved by reparametrizing the curve, if necessary \cite{SWPD}, it is not uncommon to assume that one works with proper parametrizations. However, we do not need to make any a priori assumption on the properness of $\phi(t)$. 

By reducing the components of $\phi(t)$ to common denominator, we can write
\begin{equation} \label{curve}
\phi(t)=\left(\mathcal{X}(t),\mathcal{Y}(t)\right)=\left(\frac{X(t)}{W(t)},\frac{Y(t)}{W(t)}\right),
\end{equation}
where $\gcd(X(t),Y(t),W(t))=1$. Calling
\begin{equation}\label{UV}
U(t)=X'(t)W(t)-X(t)W'(t),\mbox{ }V(t)=Y'(t)W(t)-Y(t)W'(t),
\end{equation} 
we define the \emph{offset} to ${\mathcal C}$ at distance $d\in {\Bbb R}^+$, ${\mathcal O}_d({\mathcal C})$, as the \emph{Zariski closure} of the set of points $(x,y)=\phi_d(t)$, where 
\begin{equation} \label{offset}
\phi_d(t)=\left(\frac{X(t)}{W(t)}\pm d \frac{V(t)}{\sqrt{U^2(t)+V^2(t)}},\frac{Y(t)}{W(t)}\mp d\frac{U(t)}{\sqrt{U^2(t)+V^2(t)}}\right),
\end{equation}
with $U^2(t)+V^2(t)\neq 0$, $W(t)\neq 0$. Hence, $(x,y)=\phi_d(t)$ implies that the Euclidean distance between $(x,y)\in {\mathcal O}_d({\mathcal C})$ and the point $p=\phi(t)\in {\mathcal C}$, measured along the normal line to ${\mathcal C}$ through $p=\phi(t)$, is $d$; in this situation, we say that $p=\phi(t)$ {\it generates} $(x,y)$. By an abuse of the language, we will also say that $(x,y)\in {\mathcal O}_d({\mathcal C})$ is generated by $t$.
Furthermore, when the first sign of $\pm$ and $\mp$ in the expression \eqref{offset} is considered, the geometrical locus described is called the \emph{exterior offset}; if the second sign is chosen, the geometrical locus described this way is called the \emph{interior offset}. The union of the exterior and interior offsets is the whole offset, ${\mathcal O}_d({\mathcal C})$.

 In general, we will say that $P\in {\mathcal C}$ generates $P_d\in {\mathcal O}_d({\mathcal C})$ if $P_d=P\pm d\cdot {\mathcal N}(P)$, where ${\mathcal N}(P)$ is the unitary normal vector to ${\mathcal C}$ at $P$. If $\mbox{lim}_{t\to \infty}\phi(t)$ is an affine point, which happens iff $\deg(X(t))\leq \deg(W(t))$ and $\deg(Y(t))\leq \deg(W(t))$, then $\mbox{lim}_{t\to \infty}\phi_d(t)$ gives rise to two points, which we denote by $P_{\pm \infty}=(x_{\pm \infty},y_{\pm \infty})$. The points $P_{\pm \infty}$ also belong to ${\mathcal O}_d({\mathcal C})$. 

The computation of the implicit equation of ${\mathcal O}_d({\mathcal C})$ is addressed in \cite{Far2}. In order to compute this equation, the following polynomials are introduced: 
\begin{eqnarray} 
\tilde{P}(x,y,t):= U(t)( W(t)x-X(t) ) +V(t)( W(t)y-Y(t) ) = 0, \label{P}\\
\tilde{ Q}(x,y,t):= ( W(t) x - X(t) )^2 + ( W(t)y - Y(t) )^2 - d^2 W^2(t) = 0.\label{Q}  
\end{eqnarray}
%
Roughly speaking, the implicit equation  of ${\mathcal O}_d({\mathcal C})$ is found by eliminating the variable $t$ in the system formed by \eqref{P} and \eqref{Q}. However, in order to avoid as many extraneous components as possible, we first divide $\tilde{P}(x,y,t)$, $\tilde{Q}(x,y,t)$ by their \emph{contents} with respect to $t$.\footnote{Let $f(x_1,\ldots,x_r,x_{r+1},\ldots,x_n)$ be a polynomial in the variables $x_1,\ldots,x_r,x_{r+1},\ldots,x_n$ with coefficients in a unique factorization domain. The \emph{content} $\mbox{cont}_{x_1,\ldots,x_r}(f)$ of $f$ with respect to $x_1,\ldots,x_r$ is the $\gcd$ of the coefficients of $f$, seen as a polynomial in $x_{r+1},\ldots,x_n$ whose coefficients are polynomials in $x_1,\ldots,x_r$. The polynomial $\tilde{f}=\frac{1}{\mbox{cont}_{x_1,\ldots,x_r}(f)}\cdot f$ is called the \emph{primitive part} of $f$ with respect to $x_1,\ldots,x_r$.} The polynomials obtained from $\tilde{P}$ and $\tilde{Q}$ after removing their $t$-contents are denoted by $P(x,y,t)$ and $Q(x,y,t)$, respectively.

Now writing 
\begin{equation}\label{lares}
H(x,y)=\mbox{Res}_t(P(x,y,t),Q(x,y,t)),
\end{equation}
we have (see Theorem 3.6 of \cite{Far2}) that 
\begin{equation}\label{structure}
H(x,y)=F(x,y)\cdot G(x,y),
\end{equation}
where $F(x,y)$ is \emph{an implicit equation} of ${\mathcal O}_d({\mathcal C})$, and $G(x,y)$ is a product of extraneous linear factors. The extraneous linear factors can be computed a priori by using Lemma 3.4 of \cite{Far2}. It is worth observing that although \cite{Far2} assumes that $\phi(t)$ is proper, by following the discussion in \cite{Far2} one can check that the results in Theorem 3.6 and Lemma 3.4 of \cite{Far2} are also valid even if $\phi(t)$ is not proper. 

Notice that we said that $F(x,y)$ is ``an" implicit equation, and not ``the" implicit equation. The reason is that $h(x,y)=0$ is the implicit equation of a curve iff $h(x,y)$ is the polynomial \emph{of minimum degree} implicitly defining the curve. Hence, $F(x,y)$ is ``the" implicit equation of ${\mathcal O}_d({\mathcal C})$ iff it is square-free. We will prove, in Section \ref{sec-main}, that this is certainly the case whenever two conditions, that we will make precise, hold; if some of these conditions fails then $F(x,y)$ is not square-free, although we will see that the multiplicity of its components can be computed in advance. 

The first of these conditions is related to the properness of $\phi(t)$, that we recalled at the beginning of the section. The second condition has to do with the notions of {\it simple} and {\it special} components of an offset curve. These notions were introduced in \cite{Rafa1}, where a more algebraic definition of the offset curve, using an incidence diagram, is given. Based on this incidence diagram, it is proven that ${\mathcal O}_d({\mathcal C})$ has at most two components. Furthermore, an irreducible component of ${\mathcal O}_d({\mathcal C})$ is said to be \emph{simple} if almost every point of that component is generated by just one point of ${\mathcal C}$; otherwise, the component is called \emph{special}. In \cite{Rafa1} it is proven that for almost all $d$, ${\mathcal O}_d({\mathcal C})$ has simple components, and that ${\mathcal O}_d({\mathcal C})$ can have at most one special component. Furthermore, it is also demonstrated that if ${\mathcal O}_d({\mathcal C})$ is irreducible then it is simple, and that special components only appear when computing offsets to offsets (in which case the special component is the original curve).

Additionally, at certain moments in the paper we will distinguish two different types of singular points, {\it local singularities} and {\it self-intersections}. In the first case there is just one branch of the curve going through the point, while in the second case there are at least two different branches of the curve through the point (see Section 2.3 of \cite{paper} for more information). Notice that the number of local singularities of an algebraic curve is always finite. Furthermore, if an algebraic curve is defined by a square-free polynomial then it also has finitely many self-intersections.

We finish with two results related to the preceding notions, that we will use later. The first one is related to the notions of simple and special components of ${\mathcal O}_d({\mathcal C})$.

\begin{lemma}\label{nocircle}
There are finitely many points of ${\mathcal O}_d({\mathcal C})$ generated by more than two points of ${\mathcal C}$.
Furthermore, a component ${\mathcal V}_d$ of ${\mathcal O}_d({\mathcal C})$ is special iff almost all points of ${\mathcal V}_d$ are generated by exactly two points of ${\mathcal C}$.
\end{lemma}

\begin{proof} Let $F^{\star}(x,y)$ be the square-free part of $F(x,y)$, and let $P_d\in {\mathcal O}_d({\mathcal C})$ be a regular point of the curve $F^{\star}(x,y)=0$, which also defines the offset. The points $P\in {\mathcal C}$ generating $P_d$ are the intersection points of each line ${\mathcal L}_{P_d}$, normal to ${\mathcal O}_d({\mathcal C})$ through $P_d$, the circle $C_d$ centered at $P_d$ of radius $d$, and ${\mathcal C}$. Since by hypothesis $P_d$ is regular there is just one normal line ${\mathcal L}_{P_d}$. Since the intersection of ${\mathcal L}_{P_d}$ and $C_d$ consists of at most two points, and since the curve $F^{\star}(x,y)=0$ has finitely many singularities, the first part holds. The second part follows from the notion of special component. 
\end{proof}

Finally, we recall the next lemma, which is proven in Appendix I of \cite{ACD15}; although in \cite{ACD15} one works with properly parametrized curves, one can check that the proof of the lemma does not depend on the properness of the curve, and therefore it is also valid in the case of non-proper curves.
\begin{lemma}\label{l-aux}
The only points of the offset where the leading coefficients of $P(x,y,t)$, $Q(x,y,t)$ with respect to $t$ simultaneously vanish are $P_{\pm\infty}$, in the case when $P_{\pm\infty}$ are affine points.  
\end{lemma}

\subsection{Factorization of resultants}
Let ${\mathcal M}_p$, ${\mathcal M}_q$ be two algebraic curves without common components, and let $p(x,y)$, $q(x,y)$ be the polynomials implicitly defining ${\mathcal M}_p$ and ${\mathcal M}_q$. Since ${\mathcal M}_p$, ${\mathcal M}_q$ do not have common components, ${\mathcal M}_p$, ${\mathcal M}_q$ intersect at finitely many points. Furthermore, each intersection point has an {\it intersection multiplicity}. The notion of intersection multiplicity is described, for instance, in \S IV.5.1 \cite{walker}, or in \S 1.6 \cite{Fulton}. In the general case, computing the intersection multiplicity requires certain technicalities; however, when the intersection point is regular for both ${\mathcal M}_p$ and ${\mathcal M}_q$ and the tangent line at such a point is different for ${\mathcal M}_p$ and ${\mathcal M}_q$, the intersection multiplicity is 1. Now let 
\[\mbox{Res}_y(p(x,y),q(x,y))=c\sum_{i=1}^r (x-\alpha_i)^{\beta_i},\]where $c$ is constant, and the $\alpha_i$ are the (possibly complex) roots of $\mbox{Res}_y(p,q)$. Then (see Proposition 5 \cite{Buse}, and \S 1.6 \cite{Fulton}) we have the following result.

\begin{proposition} \label{prop-aux}
Suppose that $x=\alpha_i$ is not a common vertical asymptote of ${\mathcal M}_p$ and ${\mathcal M}_q$. For all $i=1,\ldots,r$, the integer $\beta_i$ equals the sum of all the intersection multiplicities of the points $z_j = (x_j, y_j)\in {\mathcal M}_p \cap {\mathcal M}_q$ such that $x_j=\alpha_i$. In particular, if the line $x=\alpha_i$ does not contain any point of ${\mathcal M}_p \cap {\mathcal M}_q$ which is singular for either ${\mathcal M}_p$ or ${\mathcal M}_q$, and no point of ${\mathcal M}_p\cap{\mathcal M}_q$ where the tangent lines to ${\mathcal M}_p$ and ${\mathcal M}_q$ coincide, then $\beta_i$ is equal to the number of (possibly complex) different points of ${\mathcal M}_p \cap {\mathcal M}_q$ lying on the line $x=\alpha_i$. 
\end{proposition}

\section{Square-freeness of $F(x,y)$.} \label{sec-main}

In this section we will give a complete description of the factorization of  $F(x,y)$ (see Eq. (\ref{structure})). As a consequence, we will provide necessary and sufficient conditions for $F(x,y)$ to be square-free. We will need first the following result.   

\begin{lemma} \label{aux2}
Let ${\mathcal C}$ be a curve parametrized by $\phi(t)$, not necessarily proper. Let $P_t(x,y,t)$ denote the partial derivative of $P(x,y,t)$ with respect to the variable $t$. Then there are just finitely many points $(x,y)\in {\mathcal O}_d({\mathcal C})$ such that there exists $t\in {\Bbb C}$ satisfying
\begin{equation}\label{pqt}
P(x,y,t)=Q(x,y,t)=P_t(x,y,t)=0.
\end{equation}
\end{lemma}

\begin{proof} Let $\phi_d(t)=(x_d(t),y_d(t))$ be the points of the offset generated by the parameter $t$ (see Eq. \eqref{offset}). Hence, we have that 
\begin{equation}\label{pq}
P(x_d(t),y_d(t),t)=Q(x_d(t),y_d(t),t)=0,
\end{equation}
i.e. the functions $P(x_d(t),y_d(t),t),Q(x_d(t),y_d(t),t)$ are identically zero. Now assume that the statement is false. Then the function 
$P_t(x_d(t),y_d(t),t)$ is also identically zero. Therefore, by arguing as in Lemma 8 of \cite{paper}, we deduce that ${\mathcal O}_d({\mathcal C})$ has infinitely many local singularities, which cannot be.
\end{proof}

Now we will separately consider the case when the parametrization $\phi(t)$ is proper, and the case when its tracing index is bigger than $1$. 

\subsection{Case of proper parametrizations.}\label{prop-case}

In this subsection we will assume that the parametrization $\phi(t)$ is proper. Recalling the notation in Eq. \eqref{structure} and since ${\mathcal O}_d({\mathcal C})$ has at most two components, we can write 
\begin{equation} \label{F}
F(x,y)=\left(f_1(x,y)\right)^r\cdot \left(f_2(x,y)\right)^s,
\end{equation}
where $f_1(x,y)$, $f_2(x,y)$ are irreducible, therefore square-free, and $f_1(x,y)$ is non-constant. If ${\mathcal O}_d({\mathcal C})$ has just one component, then $f_2(x,y):=1$. Furthermore, if ${\mathcal O}_d({\mathcal C})$ is reducible and ${\mathcal O}_d({\mathcal C})$ has some special component (recall that there is at most one component of this type), we will assume that $f_2(x,y)$ implicitly represents this special component.   

\begin{lemma}\label{line}
The polynomial $f_1(x,y)$ explicitly depends on $x$ and $y$. If $f_2(x,y)$ is not constant, then it also explicitly depends on $x$ and $y$.
\end{lemma}

\begin{proof} We prove the result for $f_1(x,y)$; similarly for $f_2(x,y)$, in the case when it is not constant. Suppose that $f_1(x,y)$ depends on just one variable, $x$ or $y$. Then ${\mathcal O}_d({\mathcal C})$ contains one line. Since ${\mathcal C}\subset {\mathcal O}_d({\mathcal O}_d({\mathcal C}))$ and the offset of a line is a pair of two parallel lines, ${\mathcal C}$ contains at least one line. On the other hand, since ${\mathcal C}$ is rational by hypothesis, it is irreducible. Hence, if ${\mathcal C}$ contains a line then ${\mathcal C}$ must be a line. However, this is impossible because by hypothesis ${\mathcal C}$ is not a line. 
\end{proof}

Now let us consider the set ${\mathcal A}$ of the $y_0$s satisfying some of the following conditions:

\begin{itemize}
\item [(1)] The intersection of the line $y=y_0$ with ${\mathcal O}_d({\mathcal C})$ contains some point also belonging to the curve $G(x,y)=0$.
\item [(2)] The leading coefficients of $P(x,y,t)$ and $Q(x,y,t)$ with respect to $t$ identically vanish when $y=y_0$. 
\item [(3)] The line $y=y_0$ contains a local singularity of ${\mathcal O}_d({\mathcal C})$. 
\item [(4)] The line $y=y_0$ is tangent to ${\mathcal O}_d({\mathcal C})$.
\item [(5)] There exist $x_0,t_0$ such that Eq. \eqref{pqt} holds for $(x_0,y_0,t_0)$.
\item [(6)] The line $y=y_0$ contains either a point of a simple component of ${\mathcal O}_d({\mathcal C})$ generated by more than one value of $t$, or a point of a special component of ${\mathcal O}_d({\mathcal C})$ generated by more than two values of $t$.

\end{itemize}

Notice that in particular, the intersection points between two components of ${\mathcal O}_d({\mathcal C})$ satisfy condition (6). Furthermore, from Lemma \ref{l-aux} we have that condition (2) implies $y_0= y_{\pm \infty}$. Then we have the following result.

\begin{lemma} \label{finite}
${\mathcal A}$ is a finite set.
\end{lemma}

\begin{proof} It is clear that there are finitely many $y_0$s satisfying (1), (2), (3), (4) and (5). So let us see that there are also finitely many $y_0$s satisfying (6). If ${\mathcal O}_d({\mathcal C})$ does not have any special component, since by assumption ${\mathcal C}$ is properly parametrized there are finitely many points of ${\mathcal O}_d({\mathcal C})$ generated by more than one value of the parameter $t$. If ${\mathcal O}_d({\mathcal C})$ has a special component ${\mathcal V}$, by Lemma \ref{nocircle} a generic point of ${\mathcal V}$ is generated by two points of ${\mathcal C}$; since ${\mathcal C}$ is by assumption properly parametrized, this means that a generic point of ${\mathcal V}$ is generated by exactly two values of $t$.  
\end{proof}

Therefore, a \emph{generic} $y_0$ does {\it not} satisfy any condition (1)--(6). This is crucial in the next theorem, which is our first important result. 

\begin{theorem} \label{libre}
Suppose that ${\mathcal C}$ is properly parametrized. 
\begin{itemize}
\item [(1)] If ${\mathcal O}_d({\mathcal C})$ does not have any special component, then \[F(x,y)=f_1(x,y)\cdot f_2(x,y).\]
\item [(2)] If ${\mathcal O}_d({\mathcal C})$ has a special component, then \[F(x,y)=f_1(x,y)\cdot \left(f_2(x,y)\right)^2.\]
\end{itemize}
\end{theorem}

\begin{proof} (1) Let $F(x,y)=\left(f_1(x,y)\right)^{r_1}\cdot \left(f_2(x,y)\right)^{r_2}$. We want to prove that $r_1=r_2=1$. By Lemma \ref{line} $f_1(x,y)$ and $f_2(x,y)$, when it is not constant, explicitly depend on $x$ and $y$. Therefore it suffices to show that for a generic $y_0$, $f_1(x,y_0)$ and $f_2(x,y_0)$ are square-free. Let $y=y_0$ be generic, so that $y_0$ does not satisfy any condition (1)--(6). In particular, the leading coefficients of $P(x,y,t)$ and $Q(x,y,t)$ do not identically vanish for $y=y_0$, so $\mbox{Res}_t(P(x,y,t),Q(x,y,t))$ specializes properly, i.e. \[H(x,y_0)=\mbox{Res}_t(P(x,y_0,t),Q(x,y_0,t))\](see Lemma 4.3.1 in \cite{winkler}). Since $y_0$ does not satisfy condition (1), the line $y=y_0$ does not intersect $H(x,y)=0$ in any point both belonging to ${\mathcal O}_d({\mathcal C})$ and to the curve $G(x,y)=0$. Moreover $y_0$ does not satisfy conditions (3), (4), (6) either, and therefore we have that 
\begin{equation}\label{H}
H(x,y_0)=(x-x_1)^{r_1}\cdots (x-x_m)^{r_1}\cdot(x-\tilde{x}_1)^{r_2}\cdots (x-\tilde{x}_n)^{r_2}\cdot G(x,y_0),
\end{equation}
where the intersections between $f_1(x,y)=0$ and $y=y_0$ correspond to the $x_i$s, the intersections between $f_2(x,y)=0$ and $y=y_0$ correspond to the $\tilde{x}_j$s, and $G(x_i,y_0)\neq 0$, $G(\tilde{x}_j,y_0)\neq 0$ for $i=1,\ldots,m$, $j=1,\ldots,n$. Now assume that $r_\ell>1$, where $f_\ell(x,y)$ is nonconstant. We fix $\ell=1$; whenever $f_2(x,y)$ is nonconstant, one can argue in the same way for $\ell=2$. From Proposition \ref{prop-aux} we have the following possibilities:
\begin{itemize}
\item [(i)] For any $i=1,\ldots,m$ there are distinct $t_0,t_1$ with $P(x_i,y_0,t_0)=Q(x_i,y_0,t_0)=0$ and $P(x_i,y_0,t_1)=Q(x_i,y_0,t_1)=0$. But this implies that $(x_i,y_0)\in{\mathcal O}_d({\mathcal C})$ is simultaneously generated by $t_0$ and $t_1$, which cannot happen because $y_0$ does not satisfy condition (6). 
\item [(ii)] There is some $i=1,\ldots,m$ such that $P(x_i,y_0,t)$ and $Q(x_i,y_0,t)$ share a root $t_0$ of multiplicity at least 2. However, in that case $P_t$ vanishes at the point $(x_i,y_0,t_0)$. And this cannot happen because $y_0$ does not satisfy condition (5).   
\item [(iii)] The line $x=x_i$ is a common vertical asymptote of the curves (defined on the $xt$-plane) $P(x,y_0,t)=0$, $Q(x,y_0,t)=0$. But this cannot happen because $y_0$ does not satisfy condition (2), and hence $y_0\neq y_{\pm \infty}$.
\end{itemize}
So we conclude that $r_1=1$. Similarly for $r_2$, whenever $f_2(x,y)$ is not constant.

Let us address now the statement (2) of the theorem. First, one argues in the same way to reach Eq. \eqref{H}. Also, one can prove that $r_1=1$ as in statement (1). In order to prove that $r_2=2$, since $y=y_0$ does not satisfy condition (6), from Lemma \ref{nocircle} one has that every point of the special component of ${\mathcal O}_d({\mathcal C})$, lying on the line $y=y_0$, is generated by exactly two values of the parameter $t$. Therefore, by Proposition \ref{prop-aux} one has $r_2\geq 2$. Now assume that $r_2>2$. Again from Proposition \ref{prop-aux} we have the following possibilities:
\begin{itemize}
\item [(i)] For any $i=1,\ldots,n$ there are at least three different values $t_0,t_1,t_2$ with $P(x_i,y_0,t_k)=Q(x_i,y_0,t_k)=0$ for $k=1,2,3$. But this implies that $(x_i,y_0)\in{\mathcal O}_d({\mathcal C})$ is simultaneously generated by $t_0$, $t_1$, $t_2$, which cannot happen because $y_0$ does not satisfy condition (6). 
\item [(ii)] There is some $i=1,\ldots,n$ such that $P(x_i,y_0,t)$ and $Q(x_i,y_0,t)$ share a root $t_0$ of multiplicity at least 2. However, in that case $P_t$ vanishes at the point $(x_i,y_0,t_0)$. But this cannot happen because $y_0$ does not satisfy condition (5).   
\item [(iii)] The line $x=x_i$ is a common vertical asymptote of the curves (defined on the $xt$-plane) $P(x,y_0,t)=0$, $Q(x,y_0,t)=0$. But this cannot happen because $y_0$ does not satisfy condition (2), and hence $y_0\neq y_{\pm \infty}$.
\end{itemize}
So we conclude that $r_2=2$.

\end{proof}

\begin{figure}
\begin{center}
\includegraphics[scale=0.6]{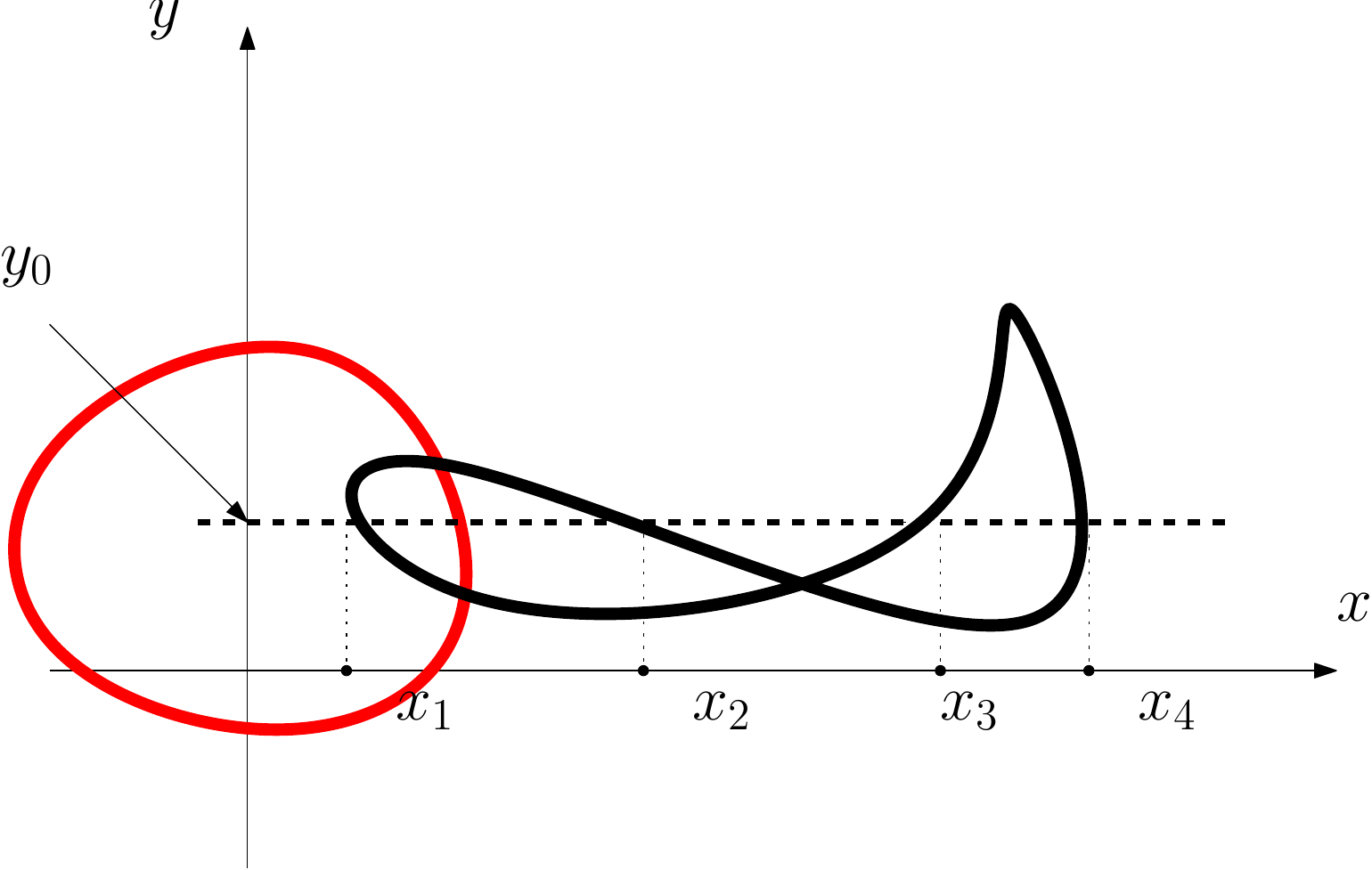}
\end{center}
\caption{The curves $G(x,y)=0$ (in red) and $f_1(x,y)=0$ (in black). The line $y=y_0$ intersects the curve $f_1(x,y)=0$ in four points, with $x$-coordinates $x_1,\ldots,x_4$}\label{fig:component}
\end{figure}

From Theorem \ref{libre} above and from Theorem 7 in \cite{Rafa1}, one gets the following result, which provides a condition for checking if a rational curve ${\mathcal C}$ is the offset to another curve (equivalently, if ${\mathcal O}_d({\mathcal C})$ has some special component).

\begin{theorem}\label{charact}
Let ${\mathcal C}$ be a rational curve properly parametrized by $\phi(t)$. The following statements are equivalent.
\begin{itemize}
\item [(1)] $F(x,y)$ is not square-free.
\item [(2)] ${\mathcal O}_d({\mathcal C})$ has a special component.  
\item [(3)] ${\mathcal C}$ is the (whole) offset to some other algebraic curve; namely, the special component of ${\mathcal O}_d({\mathcal C})$.
\end{itemize}
\end{theorem}

\begin{corollary} \label{corolsq}
If $\mbox{Res}_t(P(x,y,t),Q(x,y,t))$ is square-free, then: (1) ${\mathcal O}_d({\mathcal C})$ does not have any special component; (2) ${\mathcal C}$ is not the offset of any algebraic curve.
\end{corollary}

\subsection{Case of non-necessarily proper parametrizations.}\label{noproper}

The aim of this section is to generalize Theorem \ref{libre} to the case of non-proper parametrizations, and to characterize from here the square-freeness of $F(x,y)$; hence, in this subsection we will assume that ${\mathcal C}$ is parametrized by a non-necessarily proper rational parametrization $\phi(t)$, with tracing index $n\geq 1$. We can write $F(x,y)$ as in Eq. \eqref{F}; furthermore if ${\mathcal O}_d({\mathcal C})$ has any special component we assume that it is represented by $f_2(x,y)$. We also need to introduce a set ${\mathcal A}$ of $y_0$s satisfying six conditions. The conditions (1)-(5) coincide with those in Section \ref{prop-case}; however, condition (6) is replaced by the following condition:

\begin{itemize}
\item [(6)] The line $y=y_0$ contains either a point of a simple component of ${\mathcal O}_d({\mathcal C})$ generated by more than $n$ values of $t$, or a point of a special component of ${\mathcal O}_d({\mathcal C})$ generated by more than $2n$ values of $t$. 
\end{itemize}

\noindent One can prove that ${\mathcal A}$ is finite in an analogous way to Lemma \ref{finite}.

Now we have the following theorem, which is a generalization of Theorem \ref{libre}. It can be proven in an analogous way to Theorem \ref{libre}. In the proof, that we ommit here, one needs to observe that since the tracing index of $\phi(t)$ is $n$, a generic point of a simple component of ${\mathcal O}_d({\mathcal C})$ comes from just one point of ${\mathcal C}$, i.e. it is generated by $n$ values of the parameter $t$. Also, a generic point of the special component of ${\mathcal O}_d({\mathcal C})$, if any, comes from exactly two points of ${\mathcal C}$, i.e. it is generated by $2n$ values of the parameter $t$. Note also that the lines $y=y_0$ containing the self-intersections of ${\mathcal O}_d({\mathcal C})$ satisfy condition (6). 

\begin{theorem} \label{libre2}
Suppose that ${\mathcal C}$ is parametrized by a rational parametrization of tracing index $n$.
\begin{itemize}
\item [(1)] If ${\mathcal O}_d({\mathcal C})$ does not have any special component, then \[F(x,y)=\left(f_1(x,y)\right)^n\cdot \left(f_2(x,y)\right)^{n}.\]
\item [(2)] If ${\mathcal O}_d({\mathcal C})$ has a special component, then \[F(x,y)=\left(f_1(x,y)\right)^n\cdot \left(f_2(x,y)\right)^{2n}.\]
\end{itemize}
\end{theorem}

One can notice the analogy between the formulae in Theorem \ref{libre2}, and Eq. \eqref{res}. This was somehow expectable, since tracing the curve ${\mathcal C}$ $n$ times implies tracing also $n$ times each simple component of ${\mathcal O}_d({\mathcal C})$, and $2n$ times each special component of ${\mathcal O}_d({\mathcal C})$.  Finally we conclude with the following result, which can be derived from Theorem \ref{libre2}.

\begin{corollary}\label{corfinal}
The polynomial $F(x,y)$ is square-free iff the curve ${\mathcal C}$ is properly parametrized, and ${\mathcal O}_d({\mathcal C})$ does not have any special component.
\end{corollary}

\subsection{Examples.} \label{subsec-ex}

In this subsection we provide some examples illustrating the results in the two preceding subsections.

\begin{example}
Let ${\mathcal C}$ be the Cardioid, of implicit equation
$${x}^{4}+2\,{x}^{2}{y}^{2}+{y}^{4}+8\,{x}^{2}y+8\,{y}^{3}-16\,{x}^{2}=0.$$
This curve is rational and can be properly parametrized by 
$$
\left(\mathcal{X}(t),\mathcal{Y}(t)\right)=\left(\frac{X(t)}{W(t)},\frac{Y(t)}{W(t)}\right) =
 \left(\frac{-1024\,{t}^{3}}{ \left( 16\,{t}^{2}+1 \right) ^{2}},\frac{-128\,{t}^{2} \left( 16\,{t}^{2}-1\right) }{ \left( 16\,{t}^{2}+1 \right) ^{2}}\right) .
$$
Let us find the offset of ${\mathcal C}$ for $d=1$; for simplicity, we will denote this offset by ${\mathcal Z}$, i.e. ${\mathcal Z}={\mathcal O}_1({\mathcal C})$. By computing the resultant $H(x,y)$ in \eqref{lares}, we get that $H(x,y)$ is, up to a constant, 
the product of $F(x,y)$, 
 
\begin{eqnarray*}
F(x,y) &= &   {x}^{8}+4\,{x}^{6}{y}^{2}+6\,{x}^{4}{y}^{4}+4\,{x}^{2}{y}^{6}+{y}^{8}+
16\,{x}^{6}y+48\,{x}^{4}{y}^{3}+48\,{x}^{2}{y}^{5}+ \\
& & 16\,{y}^{7}-35\,{x}^{6}-9\,{x}^{4}{y}^{2}+87\,{x}^{2}{y}^{4}+61\,{y}^{6}-292\,{x}^{4}y-
328\,{x}^{2}{y}^{3}-\\
& & 36\,{y}^{5}+211\,{x}^{4}-234\,{x}^{2}{y}^{2}-189\,
{y}^{4}-40\,{x}^{2}y-232\,{y}^{3}-429\,{x}^{2}+\\ 
&&131\,{y}^{2}+316\,y+252 
\end{eqnarray*}
and an ``extraneous'' factor $G(x,y)$,
 $$G(x,y) = {x}^{2}+{y}^{2}+4\,y+4.$$
In this case, one can check that the curve defined by $F(x,y)$ is square-free and irreducible, so ${\mathcal Z}$ has just one, simple, component. Following the notation of Theorem \ref{libre}, we have $F(x,y)=f_1(x,y)\cdot f_2(x,y)$ with $f_1(x,y)$ square-free, and $f_2(x,y)=1$. The Cardioid, together with its offset for $d=1$, are shown in Figure \ref{fig:offcard}. 

\begin{figure}
\begin{center}
\includegraphics[scale=0.4]{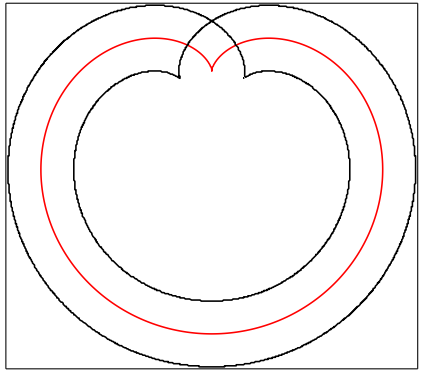}
\end{center}
\vspace{-5cm}
\caption{The Cardioid (in red) and the offset for $d=1$ (in black) . }\label{fig:offcard}
\end{figure}

\bigskip
Now let us study the offset of ${\mathcal Z}$ for $d=1$. According to Theorem 7 in \cite{Rafa1}, the offset of ${\mathcal Z}$ for $d=1$ must have a special component, namely the Cardioid. One can check that ${\mathcal Z}$ is rational; a proper parametrization of ${\mathcal Z}$ is $({\mathcal X}^{\star}(t),{\mathcal Y}^{\star}(t))$, where 
\begin{equation}\label{lapara}
\begin{array}{rcl}
 {\mathcal X}^{\star}(t) &=&    \displaystyle{\frac{ \left( {t}^{2}-9 \right)  \left( {t}^{6}-117\,{t}^{4}-1053\,{t}^{2
}+3456\,{t}^{3}+729\right) }{  \left( {t}^{2}+9 \right) ^{4}} },\\ 
{\mathcal Y}^{\star}(t) &=& \displaystyle{\frac{-18\, \left( {t}^{6}-16\,{t}^{5}-21\,{t}^{4}+864\,{t}^{3}-189\,{t}^{2}
-1296\,t+729 \right) t
 }{ \left( {t}^{2}+9 \right) ^{4} }}   .
\end{array}
\end{equation}
After computing the resultant $H(x,y)$ in Eq. \eqref{lares} for the above parametrization \eqref{lapara}, we check that $H(x,y)=F(x,y)\cdot G(x,y)$, where $G(x,y)=({x}^{2}+{y}^{2}+4\,y+4)^2$ and $F(x,y)$ factorizes into two polynomials. As predicted by statement (2) of Theorem \ref{libre}, we have 
$$F(x,y)=f_1(x,y)\cdot f_2(x,y)^2,$$
where $f_1(x,y)$ and  $f_2(x,y)$, with
\begin{eqnarray*}
f_1(x,y) &= &{x}^{8}+4\,{x}^{6}{y}^{2}+6\,{x}^{4}{y}^{4}+4\,{x}^{2}{y}^{6}+{y}^{8}+
16\,{x}^{6}y+48\,{x}^{4}{y}^{3}+48\,{x}^{2}{y}^{5}+\\
 &&16\,{y}^{7}-44\,{x}
^{6}-36\,{x}^{4}{y}^{2}+60\,{x}^{2}{y}^{4}+52\,{y}^{6}-400\,{x}^{4}y-
544\,{x}^{2}{y}^{3}-\\ 
&&144\,{y}^{5}+112\,{x}^{4}-864\,{x}^{2}{y}^{2}-720
\,{y}^{4}+128\,{x}^{2}y-640\,{y}^{3}-768\,{x}^{2}+\\&&2048\,{y}^{2}+4864\,y+3840,\\
 f_2(x,y) &= & {x}^{4}+2\,{x}^{2}{y}^{2}+{y}^{4}+8\,{x}^{2}y+8\,{y}^{3}-16\,{x}^{2}. \\
\end{eqnarray*}
Observe that, as expected, the curve defined by $f_2(x,y)$ is our initial curve, the Cardiod. Figure \ref{fig:offoffcard} shows the curve ${\mathcal Z}$ (i.e. ${\mathcal O}_1({\mathcal C}))$ in red color, and ${\mathcal O}_1({\mathcal Z})$ in black color. Moreover, the special component is plotted in dash-dotted line. 
 \begin{figure}
\begin{center}
\includegraphics[scale=0.4]{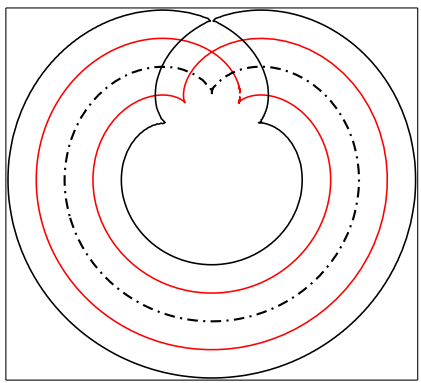}
\end{center}
\vspace{-5cm}
\caption{For  $d=1$, the offset of the Cardioid (in red), and the offset of the offset of the Cardioid (in black) . }\label{fig:offoffcard}
\end{figure}

Finally, let us reparametrize the curve $\mathcal{Z}$ by using the following non-linear change of parameter:
$$
t = s^3-2\,s^2+3\,s+5.
$$
By applying this change of parameter, we get a non-proper parametrization of ${\mathcal Z}$, of tracing index $n=3$. After computing the polynomial $F(x,y)$ with this non-proper parametrization, we find, as predicted by statement (2) of Theorem \ref{libre2}, $F(x,y)=f_1(x,y)^3\cdot f_2(x,y)^6$.

\end{example}

\begin{example} 
It is well known that the offset of the parabola is a rational curve. In this example, we analyze the offset of the offset of the parabola $y=\frac{1}{4}x^2$ for $d=6$, but considering a non-proper parametrization of the offset to $y=\frac{1}{4}x^2$. More precisely, let ${\mathcal C}$ be the curve defined by $y=\frac{1}{4}x^2$, and let ${\mathcal P}={\mathcal O}_6({\mathcal C})$. One can check that  
$$
\frac{X(t)}{W(t)} = \frac{4\, \left( {t}^{8}+6\,{t}^{6}-6{t}^{2}-1\right) {t}^{2}}{4\, \left( {t}^{4}+1 \right) {t}^{4}},
$$
$$
\frac{Y(t)}{W(t)} = \frac{{t}^{12}-{t}^{8}-48{t}^{6}-{t}^{4}+1}{4\, \left( {t}^{4}+1 \right) {t}^{4}} .
$$
is a parametrization of ${\mathcal P}$, with tracing index equal to 2. In this case, when computing the resultant $H(x,y)$ in Eq. \eqref{lares} we get the extraneous factor $ ( {x}^{2}+{y}^{2}-2\,y+1)^4$. Additionally, as predicted by Theorem \ref{libre2}, we have 
$$F(x,y)=f_1(x,y)^2\cdot f_2(x,y)^4,$$
where

\begin{eqnarray*}
f_1(x,y)&=&{x}^{6}+{x}^{4}{y}^{2}-10\,{x}^{4}y-8\,{x}^{2}{y}^{3}-431\,{x}^{4}-256\,{x}^{2}{y}^{2}+16\,{y}^{4}+\\ 
&&280\,{x}^{2}y-1184\,{y}^{3}+59328\,{x}^{2}+19600\,{y}^{2}+170496\,y-3154176,	\\
f_2(x,y)&=&x^2-4y.
\end{eqnarray*}

Figure \ref{fig:offoffpar} shows the curve ${\mathcal P}$ (the offset of $y=\frac{1}{4}x^2$ for $d=6$) in red and ${\mathcal O}_6({\mathcal P})$ in black. Moreover, the special component is plotted in dash-dotted line.

 \begin{figure}
\begin{center}
\includegraphics[scale=0.4]{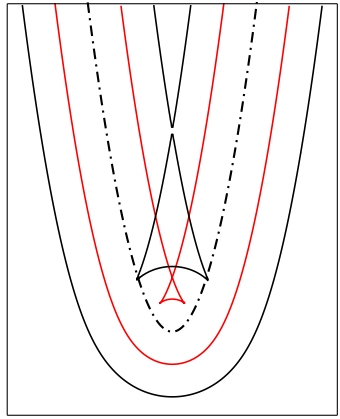}
\end{center}
\vspace{-5cm}
\caption{For $d=6$, the offset of the parabola (in red), and the offset of the offset of the parabola (in black) . }\label{fig:offoffpar}
\end{figure}
\end{example}


\section*{References}

\begin{thebibliography}{56}


\bibitem{AS07} Alc\'azar J.G., Sendra J.R. (2007), {\it Local shape of offsets to algebraic curves}, Journal of Symbolic Computation Vol. 42, pp. 338--351.

\bibitem{A08} Alc\'azar J.G. (2008), {\it Good global behavior of offsets to plane algebraic curves}, Journal of Symbolic Computation Vol. 43, Issue 9, pp. 659--680.




\bibitem{paper} Alc\'azar J.G., Diaz-Toca G.M., Caravantes J. (2015), {\it A new method to compute the singularities of offsets to rational plane curves}, Journal of Computational and Applied Mathematics, Vol. 290, pp. 385--402. 

\bibitem{ACD15} Alc\'azar J.G., Diaz-Toca G.M., Caravantes J. (2015), {\it A new method to compute the singularities of offsets to rational plane curves}. ArXiv:1502.04518.

\bibitem{AS97} Arrondo E., Sendra J.R., Sendra J. (1997), {\it Parametric Generalized Offsets to Hypersurfaces}, Journal of Symbolic Computation 
Vol. 23, Issues 2-3, pp. 267--285.

\bibitem{AS99} Arrondo E., Sendra J.R., Sendra J. (1999), {\it Genus Formula for Generalized Offset Curves}, Journal of Pure and Applied Algebra Volume Vol. 136, Issue 3, pp. 199--209.










\bibitem{Buse} Buse L., Khalil H., Mourrain B. (2005), {\it Resultant-based method for plane curves intersection problems}, Proceedings of the Conference on Computer Algebra in Scientific Computing, Lecture Notes in Computer Science, Vol. 3718, pp. 75--92, Springer. 


















\bibitem{Far2} Farouki R.T., Neff C.A. (1990), {\it Algebraic properties of plane offset curves}, Computer Aided Geometric Design Vol. 7, pp. 101-127.

\bibitem{Farouki} Farouki R.T. (2008), {\it Pythagorean-Hodograph Curves}, Springer-Verlag.

\bibitem{Fukushima} Fukushima M. (2010), {\it Hyperellipticity of offsets to rational plane curves}, Journal of Pure and Applied Algebra Vol. 214, pp. 480--492.



\bibitem{Fulton} Fulton W. (1984), {\it Introduction to intersection theory in algebraic geometry}, vol. 54 of CBMS Regional Conference Series in Mathematics. Published for the Conference Board of the Mathematical Sciences, Washington, DC (USA).






\bibitem{Kim} Kim Y.J., Lee J., Kim M.S., Elber G. (2012), {\it Efficient offset trimming for planar rational curves using biarc trees} Computer Aided Geometric Design
Vol. 29, Issue 7, pp. 555–-564.





\bibitem{Maekawa} Maekawa T., Patrikalakis N. (1993), {\it Computation of singularities and intersections of offsets of planar curves}, Computer Aided Geometric Design Vol. 10, Issue 5, pp. 407--429.













\bibitem{Rafa1} Sendra J.R., Sendra J. (2000), {\it Algebraic analysis of offsets to hypersurfaces}, Mathematische Zeitschrift Vol. 234, pp. 697-719




\bibitem{SWPD} Sendra J.R., Winkler F., P\'erez-D\'{\i}az S. (2008), {\it Rational Algebraic Curves}, Springer-Verlag.


\bibitem{Seong} Seong J.K., Elber G., Kim M.S. (2006), {\it Trimming local and global self-intersections in offset curves/surfaces using distance maps}, Computer Aided Design Vol. 38, pp. 183-193. 










\bibitem{walker} Walker R.J. (1978), {\it Algebraic curves}, Springer-Verlag.

\bibitem{winkler} Winkler F. (1996), {\it Polynomial Algorithms in Computer Algebra}. Springer Verlag, ACM Press.



\end{thebibliography}
\end{document}